\documentclass[a4paper,12pt]{article}

\usepackage{fullpage}
\usepackage{amsthm,amsmath, amssymb}
\usepackage{palatino,comment,hyperref}

\input{xy}
\xyoption{all}
\xyoption{matrix}

\usepackage{color}

\usepackage{tcolorbox}

\definecolor{verdeoscuro}{rgb}{0,0.3,0.2}

\definecolor{gris}{rgb}{0.5,0.5,0.7}

\definecolor{verdeclaro}{rgb}{0.93,1,0.9}

\definecolor{azulclaro}{rgb}{0.01,0.5,0.8}

\definecolor{celeste}{rgb}{0.02,0.7,0.9}

\definecolor{marron}{rgb}{0.3,0.1,0.0}

\newtheorem{theorem}{Theorem}[section]
\newtheorem{teo}[theorem]{Theorem}

\newtheorem{lem}[theorem]{Lemma}

\newtheorem{prop}[theorem]{Proposition}

\newtheorem{coro}[theorem]{Corollary}

\theoremstyle{definition}

\newtheorem{defi}[theorem]{Definition}

\newtheorem{rem}[theorem]{Remark}

\DeclareMathOperator{\id}{Id}

\newcommand{\wt}{\widetilde}

\def\To{\Rightarrow}

\def\t{\triangleleft}

\def\bt{\blacktriangleleft}

\def\Conj{\mathrm{Conj}}
\def\vphi{\varphi}
\def\Tw{\mathrm{Tw}}

\author{Marco A. Farinati
\thanks{Dpto de Matem\'atica FCEyN UBA - IMAS (Conicet). 
e-mail: mfarinat@dm.uba.ar.
Partially supported by 
UBACyT and
PICT 2018-00858 ``Grupos cu\'anticos, categor\'ias trenzadas e invariantes de nudos''.}
 }

\begin{document}
\title{Biracks: a notational proposal and applications}

\maketitle

\begin{abstract}
I propose a notation for biracks that includes
from the begining the knowledege of the associated 
(or underlying, or derived) rack structure. Motivated by
results of Rump in the involutive case, this notation
allows to generalize some results from involutive case
 to the  non necessarily
involutive solutions, and also to view  some twisting constructions
and its relation to the underlying rack structure
in a more transparent way. Two applications are given.
\end{abstract}


\section*{Introduction and conventions}

When considering  
set theoretical solutions of the Braid equation, 
the rack-type solutions are the maps of the form
\[
\sigma:X\times X\to X\times X\]
\[
(x,y)\mapsto (y,x\t y)
\]
where $\t:X\times X\to X$ is a binary operation.
It is an easy checking that $\sigma$ satisfies the braid equation
\[
(\sigma\times \id)(\id\times\sigma)
(\sigma\times \id)=(\id\times\sigma)
(\sigma\times \id)(\id\times\sigma)
 \]
 if and only if the operation $\t$ satisfies
 \[
(x\t y)\t z=(x\t z)\t (y\t z)
\]
That is, the operation $\t$ is self-distributive. In this case
$(X,\t)$ is called a {\em shelf}. But it is also easy to see that
\[\sigma:X\times X\to X\times X\]
\[(x,y)\mapsto (y,x\t y)\]
is bijective if and only if $(-)\t y:X\to X$
 is bijective for every $y\in X$. In such a case, $(X,\t)$ is called
 a {\em rack}. A paradigmatic example is $X=G$ a group and
\[
x\t y:=y ^{-1}xy
\]
This is called the conjugacy rack $\Conj(G)$. In this case one also see
\[
x\t x=x, \ \forall x\in X
\]
A {\em quandle} is a rack where $x\t x=x$ for all $x\in X$.
There are racks that are not quandles, an extreme case
is for example if
 $f:X\to X$ is a fixed bijection and
 \[
 x\t y:=f(x)
 \]
It is well-known that, for any rack, the natural map
\[
\vphi:X\to X\]
\[
x\mapsto\vphi(x):=x\t x
\]
is a bijection, that it is also a rack automorphims, ans that one can ``untwist'' $(X,\t)$ by $\phi$ (or twist ot by $\phi^{-1}$) and get a quandle, on the same underlying set $X$
using a new operation
\[
x\wt \t y:=\phi^{-1}(x\t y)
\]
One conclude that the set of rack structures in a 
fixed set $X$ is  ``fibered'' over the set of quandle
 opperations.

For general map $\sigma:X\times X\to X\times X$,
if has the two coordinates that depends both in
 $x$ and $y$, then notation is not unified. 
 Usual notations are like
 \[
 \sigma(x,y)=\big(\sigma ^1(x,y),\sigma^2(x,y)\big)
=\big(g_x(y),f_y(x)\big)
=\big({} ^xy,x ^y\big)
\]
\def\os{\overline{*}}
\def\us{\underline{*}}
or also, determined by two operations $\overline{*}$
and $\us$ such that
\[
\sigma(x,x\os y)=(y,x\us y)
\]
In each case, the braid equation
\[
(\sigma\times \id)(\id\times\sigma)
(\sigma\times \id)=(\id\times\sigma)
(\sigma\times \id)(\id\times\sigma)
\]
can be written, of course, in any of the choosen
notation, but despite the fact that the {\em picture}
 of the equation is clear,
the concrete formulas in terms of the choosen notation
is complicated, and properties are not easily seen from the
 expanded formulas. 
 
For general bijective solutions of the braid equation 
$\sigma:X\times X\to X\times X$,
say
 \[
 \sigma(x,y)=\big(g_x(y),f_y(x)\big)
\]
 one says that
\begin{itemize}
\item $\sigma$ is left non degenerate if $g_x:X\to X$
is biyective for any $x\in X$.
\item $\sigma$ is right non degenerate if $f_y:X\to X$
is biyective for any $y\in X$.
\item $\sigma$ is non degenerate, or a {\em birack}, if it is both left and right non-degenerate.
\item $\sigma$ is a {\em biquandle} if it is a birack and
for any $x\in X$ there exists a unique $y=y(x)$, sometimes called $s(x)$, such that
\[
\sigma(x,s(x))=(x,s(x))
\]
One can also consider a dual definition, or dual property: 
``$\sigma$ is a birack and for any $y$ there
exists a unique $x=x(y)$ such that $\sigma(x(y),y)=(x(y),y)$'', but one can see that this property is equivalent
to the previously stated condition.

\end{itemize}

 One of the non-trivial but very useful facts
 about (left non-degenerate) set theoretical solutions 
 of the braid equation is that they always determine
 a so called {\em derived solution}, that is a rack structure
on the same underlying set $X$, associated to
 $\sigma$. In this work, I propose to use a notation 
 for $\sigma$ that takes into account from the beginning the
  associated -or derived- rack. As an application, I show a generalization of a result of Rump about right non-degenerate solutions.
  Rump's result is a characterization for  non-degenerate and {\em involutive} ones, that is, solutions $\sigma$
  of the braid equation with $\sigma^2=\id_{X\times X}$
   (or equivalently, the ones whose derived rack is trivial: $x\t y=x$ for all $x,y$). 
  Rump's statement generalizes to arbitrary (i.e. non necessarily involutive) solutions, and one gets a formula similar to Rump's one but involving also the derived rack: see
  Theorem \ref{xoverrack} and
  Theorem \ref{x2}.
  This generalization shows that the  ``extra'' non-degeneracy condition
  proposed in \cite{Nelson} (called ``diagonally bijective,
  see Definition \ref{defnelson} and discussion in
  Subsection
   \ref{sectionnelson})
  is automatically satisfied
  by the usual left and {\em right} non-degeneracy.

Also, I study two ways of twisting a solution by an automorphism, and their  relation  with the  operations
defining the solution (one of them being the rack one).
As a consequence, one can see in a transparent way that
every birack is a twisting of a biquandle,
because one can also see (in a transparent way) that 
 twist map
of a general birack is exactely the same map as the Twist
map of its derived rack. 
It worth to notice that
one of the twisting procedure I consider, 
gives the family of 
skew-racks considered in \cite{N}.
Finally, I 
compare the kind of solutions that one gets from skew braces
 (yet another notation for a special family of solutions), and
 we see in a clear way that the skew-brace solutions
  corresponds  with  solutions whose
associated rack is a group with conjugation as operation $\t$.

\section{From racks to biracks and viceversa}

In \cite{R}, among ather things,  the notion of cyclic set is introduced. 
A set $X$
together with a  binary operation $\cdot$ is called a cyclic set if
$x\cdot-:X\to X$ is a bijective for all $x\in X$,
with inverse denoted by $x*-$,  and the equality
\[
(y\cdot x)\cdot (y\cdot z)=(x\cdot y)\cdot (x\cdot z)
\]
holds for all $x,y,z\in X$. If this is the case, then the map
$\sigma:X^2\to X^2$ defined by
\[
\fbox{\fbox{$
\sigma(x,y)=
\Big(x* y\ ,\ (x*y)\cdot x \Big) 
$}}
\]
is a left non-degenerte involutive solution of the Braid equation, and reciprocally,
every left non-degenerate involutive solution of the Braid equation
is of this form. 

In this section we generalize this construction to non necesarily 
involutive solutions. Given $(X,\sigma)$ a non degenerated
solution of the Braid equation we
recall the construcion of the associated rack.
The ``remarkable map'' 
$\Phi:X\times X\to X\times X$ can be 
defined in the following way
(see \cite{S} or \cite{LYZ}): using the notation 
$\sigma(x,y)=(g_x(y),f_y(x))$, define $\Phi$  by
$\Phi(x,y):=(x,g_xy)$.
Notice that  the nondegenerated
 condition means that both $g_x$ and $f_x$ are bijective maps for any $x$, in particular $\Phi$ is bijective.
This map was considered in \cite{S}, \cite{LYZ}, and several places after.
One of its properties is that there is a commutative diagram:
\[
\xymatrix{
X\times X\ar[r]^\Phi\ar[d]^\sigma& X\times X\ar[d]^{c_\t}\\
X\times X\ar[r]^\Phi& X\times X
}
\]
where $c_\t$ is of the form $c_\t(x,y)=(y\ ,\ \dots)$,
and that $c_\t$ also verifies Braid equation (this last assertion
can be proved directly, or also can be seen as a consequence of Theorem \ref{xoverrack}). That is, $c_\t(x,y)=
(y,x\t y)$ where $\t$ is necesarily a rack operation in $X$. This rack associated to $\sigma$ is called the {\em derived rack} and in terms of $g_x(y)$ and $f_y(x)$, concretely:
\[
x\t y:=x\t_\sigma y = g_{y}(f_{g_x^{-1}(y)}(x))
\]
or, equivalently, it is determined by

\begin{equation}\label{derived}
x\t g_xy=g_{g_xy}(f_{y}x)
\end{equation}

\begin{rem}
$\sigma^2=\id_{X^2}\iff
c_\t^2=\id_{X^2}$, in particular $\sigma$ is involutive
if and only if the associated rack is trivial.
\end{rem}

\begin{rem}
In \cite[Lemma 3]{FG} it is 
proved that a birack  $\sigma$ is a biquandle
if and only if its derived rack is a quandle.
\end{rem}

\subsection{Extra operation: a notational proposal}
If $\sigma(x,y)=(g_xy,f_yx)$ is a left non degenerated solution of the Braid equation
 one can simply {\em denote}
 \[
 x*y:=g_xy
 \] 
 The (left) non degeneracy condition says that, for 
 fixed $x$, the
 formula $x*y$ dependes biyectively on $y$. Let us denote $\cdot$
 the inverse operation, that is,
\[
x\cdot z=y \iff x*y=z
\]
Recall from \eqref{derived}
$
x\t g_xy=g_{g_xy}(f_{y}x)$, so in  terms of the operations we have
\[
x\t (x*y)=g_{x*y}(f_{y}x)
=(x*y)*(f_{y}x)
\]
or equivalently
\[
f_{y}x=
(x*y)\cdot(x\t (x*y))
\]
We conclude that
\begin{equation}
\fbox{
\fbox{
$
\sigma(x,y)=(g_xy,f_yx)=
\Big(x*y,
(x*y)\cdot\big(x\t (x*y)\big)\Big)
$}}
\end{equation}
In particular, if the rack is trivial (e.g. $\sigma$ involutive)
we have $x\t (x*y)=x$, so
\[
x\t (x*y)=x
\ \ \ \To \ \ \ \ 
\sigma(x,y)=\big(x*y,(x*y)\cdot x\big)
\]
just as in Rump's description.

\begin{rem}
Since $(x,y)\mapsto (x,x\cdot y)$ is bijective, another equivalent formula for the braiding is given as evaluated in elements of that form, and sometimes gives a simplification
is computation:
\begin{equation}
\sigma(x,x\cdot y)
=\big(y,
y\cdot (x\t y)\big)
\end{equation}
\end{rem}

One may wonder, given a set $X$ together with two 
binary operations $*$ and $\t$, where $x*(-)$ is invertible for every $x$, and denoting $\cdot$ the inverse operation:
\[
x\cdot z=y\iff z=x*y,\]
which are the conditions on $\cdot$ and $\t$ such that
the formula 
\[
\sigma(x,y)=
\Big(x*y,
(x*y)\cdot\big(x\t (x*y)\big)\Big)
\]
 gives a solution of Braid equation.

The following Theorem is an answer, 
that generalizes Rump's result.
 One can found related results in the litterature,
 we mention \cite{GV} (see discusion in Section
 \ref{braces}). One can see our result as a 
 generalization of the situation considered in 
 \cite[Definition 5.1, Remark 5.2]{AV}.
 Also the
 corresponding statement for parametric solutions
is given in \cite[Theorem 2.12]{Anastasia}:

\begin{teo}
\label{xoverrack}
Let
$\sigma(x,y)=(g_xy,f_yx)$ be a left non-degenerate map in $X^2$, denote
$x*y:=g_xy$ and $x\cdot -$ the inverse of $x*-$; denote also
$x\t y:=g_y(f_{g_x^{-1}y}x)=y*(f_{x\cdot y}x)$.
Then $\sigma$ satisfy the Braid equation 
if and only if the 
 following conditions hold
 \begin{enumerate}
\item 
$
(y\cdot (x\t  y))\cdot (y\cdot  z)= 
(x\cdot y)\cdot (x\cdot z)$, $\forall x,y,z\in X$.
\item
$x\cdot-$ is
a  morpism with respect to $\t$ operation, that is
\[
(x\cdot y)\t(x\cdot z)=x\cdot(y\t z)
\]
\item $(X,\t)$ is self distributive (i.e. a shelf):
\[
(x\t y)\t z=(x\t z)\t (y\t z)
\]
\end{enumerate}
\end{teo}

The proof of this theorem is a straightforward checking,
 but for 
convinience of the reader we give a direct proof of one implication.
We begin with $(X,\t,\cdot)$ where $\t$ and $\cdot$ are
arbitrary operations satisfying the non degeneracy condition
that $-\t x$ and $x\cdot -$ are bijections in $X$. Denote
$x*-$ the inverse of $x\cdot-$ and consider
the map $\sigma:X\times X\to X\times X$ defined by
\[
\sigma(x,y)=
\Big(x*  y\ ,\ (x*  y)\cdot (x\t(x*  y)) \Big) 
\]
We will check Braid equation, but instead of checking 
on elements $(x,y,z)$, it is more convenient to check
 on elements of the form
\[
(x',y',z'):=\big(x,x\cdot y,(x\cdot y)\cdot(x\cdot z)\big)
\]
The map
$(x,y,z)\mapsto \big(x,x\cdot y,(x\cdot y)\cdot(x\cdot z)\big)$
being bijective, this checking is equivalent.

Let us denote LHS the result in $X\times X\times X$ of this
part of Braid equation:
\[
\xymatrix{
x\ar[rd]|\hole&x\cdot y\ar[ld]&(x\cdot y)\cdot (x\cdot z)\ar@{=}[d]\\
y \ar[d]& y\cdot (x\t y)\ar[rd]|\hole&z'\ar[ld]\\
y \ar[rd]|\hole& [y\cdot (x\t y)]*z'=:A\ar[ld]&A\cdot ([y\cdot (x\t y)] \t A)\ar@{=}[d]\\
y*A=:B&B\cdot (y\t B)&A\cdot ([y\cdot (x\t y)] \t A)
}
\]
While RHS is the result of
\[
\xymatrix{
x\ar[d]&x\cdot y\ar[rd]|\hole&(x\cdot y)\cdot(x\cdot z)\ar[ld]\\
x\ar[rd]|\hole&x\cdot z\ar[ld] & (x\cdot z)\cdot ((x\cdot y)\t(x\cdot z))\ar@{=}[d]\\
z\ar[d]&z\cdot (x\t z)\ar[rd]|\hole&C\ar[ld]\\
z&[z\cdot (x\t z)]*C:=D&D\cdot ( [z\cdot (x\t z)]\t D)
}
\]
The first (left most) condition in Braid equation is
\[
B=y*A=y*([y\cdot (x\t y)]*z')\overset{?}{=} z
\]
which is equivalent to
\[
z'=(y\cdot (x\t y))\cdot ( y\cdot z)
\]
Recall $z'=(x\cdot y)\cdot (x\cdot z)$, so
we get identity 1:
\[
(x\cdot y)\cdot (x\cdot z)=(y\cdot (x\t y))\cdot ( y\cdot z)
\]
Now the second (middle) condition,
since $B=z$ 
and $C= (x\cdot z)\cdot ((x\cdot y)\t(x\cdot z))$ 
is
\[
B\cdot (y\t B)=
[z\cdot (x\t z)]*C
\]
That is
\[
 z\cdot (y\t z)=
[z\cdot (x\t z)]*
[
 (x\cdot z)\cdot ((x\cdot y)\t(x\cdot z))]
\]
or equivalently
\[
[z\cdot (x\t z)]\cdot[z\cdot (y\t z)]=
 (x\cdot z)\cdot ((x\cdot y)\t(x\cdot z))
\]
But using
\[
(b\cdot a)\cdot (b\cdot c)=(a\cdot (b\t a))\cdot ( a\cdot c)
\]
 for $a=z$, $b=x$ and $c=(y\t z)$
we get that the second condition  is equivalent to
\[
(x\cdot z)\cdot (x\cdot (y\t z))
=
 (x\cdot z)\cdot \big((x\cdot y)\t(x\cdot z)\big)
\]
and after canceling $(x\cdot z)$ we get identity 2:
\[
x\cdot (y\t z)
=
(x\cdot y)\t(x\cdot z)
\]

Finally, the third (rightmost) condition is
\[
A\cdot ([y\cdot (x\t y)] \t A)
=D\cdot ( [z\cdot (x\t z)]\t D)
\]
Recall that
\[A= [y\cdot (x\t y)]*z'
=
 [y\cdot (x\t y)]*[(x\cdot y)\cdot (x\cdot z)]
\]
and using the first condition we replace
$(x\cdot y)\cdot (x\cdot z)$ and get
 \[
A=
 [y\cdot (x\t y)]*[(y\cdot (x\t y))\cdot ( y\cdot z)]
=  y\cdot z
\]
So, the 3rd coordinate of LHS is
\[
(y\cdot z)\cdot ([y\cdot (x\t y)] \t (y\cdot z))
=(y\cdot z)\cdot  (y\cdot [(x\t y)\t z])
\]
where we use that $y\cdot(-)$ is a morphism for $\t$.
For RHS, 
$D=[z\cdot (x\t z)]*C$, where 
\[
C=(x\cdot z)\cdot ((x\cdot y)\t(x\cdot z))
=(x\cdot z)\cdot (x\cdot (y \t z))
=(z\cdot (x\t z))\cdot ( z\cdot (y\t z))\]
So $
D=[z\cdot (x\t z)]*C=z\cdot (y \t  z)$.
Hence,
\[
RHS=D\cdot ( [z\cdot (x\t z)]\t D)
\]
\[
=
(z\cdot (y \t  z))\cdot ( [z\cdot (x\t z)]\t (z\cdot (y \t  z)))
\]
\[
=
(z\cdot (y \t  z))\cdot ( z\cdot [(x\t z)\t (y \t  z)] )
\]
and using again
\[
(b\cdot a)\cdot (b\cdot c)=
(a\cdot (b\t a))\cdot ( a\cdot c)
\]
for $a=z$, $b=y$ and $c=[(x\t z)\t (y \t  z)]$ we get
\[
RHS=
(y\cdot z)\cdot (y\cdot [(x\t z)\t (y \t  z)])
\]
We conclude $LHS=RHS$ if and only if
\[
(y\cdot z)\cdot  (y\cdot [(x\t y)\t z])
=
(y\cdot z)\cdot (y\cdot [(x\t z)\t (y \t  z)])
\]
that, after cancelling $(y\cdot z)\cdot -$ and $y\cdot -$, is equivalent to
\[
(x\t y)\t z
=
(x\t z)\t (y \t  z)
\]

\section{Bijectivity and Right non-degeneracy}

We emphasis that if $(X,\t)$ is a rack and $*:X\times X\to X$ 
is an operation satisfying the conditions of the previous theorem, then
\[
\sigma(x,y)=\Big(x*y,(x*y)\cdot \big(x\t (x*y)\big)\Big)
\]
is left non-degenerate, but not necesarily {\em right}
 non degenerate.
An extreme case is for trivial racks: $x\t y=x$ for all $y$ that
corresponds to the involutive case ($\sigma^2=\id$)
\[
\sigma(x,y)=\big(x*y,(x*y)\cdot x\big)
\]
considered by Rump (see \cite{R}). 
The non-degeneracy of the Rack structure do not have to do with right non-degeneracy, but rather with bijectivity of $\sigma$ as we can see in the following Lemma:

\begin{lem}
With same notations as in Theorem \ref{xoverrack}, if
 $\sigma:X\times X\to X\times X$  is a left-non degenerate
solution of the Braid equation, then the associated shelf $(X,\t)$ 
is a rack
if and only if $\sigma$ is bijective.
\end{lem}

\begin{proof}Assume $\sigma$ is bijective.
Let $x,y\in X$, we compute $\sigma^{-1}(y,y\cdot x)=:(a,b)$.
Since $a*-$ is a bijection, we can write $b=a\cdot c$ for some $c$
 and get
\[
(y,a\cdot x)=\sigma(a,b)=\sigma(a,a\cdot c)=
\big(c,c\cdot (a\t c) \big)
\]
Then necessarily $c=y$ and
\[
y\cdot (a\t y)=y\cdot x\
\ \To a\t y =x
\]
so, $-\t y$ is surjective. But also, if $a\t y = a'\t y$ then
\[
\sigma(a,a\cdot y)
=
\big(y,y\cdot(a\t y)\big)
=\big(y,y\cdot(a'\t y)\big)
=
\sigma(a',a'\cdot y)
\]
and because $\sigma$ is bijective we conclude $a=a'$. That is, $-\t y$ is also injective.

On the other direction, if $(-)\t x$ is bijective for all $x$,
denote as usual
 $x\t^{-1}y$ the unique element such that $(x\t ^{-1}y)\t y=x$.
 From
\[
\sigma(x, x\cdot y)
=\big(y,y\cdot (x\t y)\big)
=\big(y,(y\cdot x)\t (y\cdot y)\big)
\]
we also deduce

\[
\sigma\big((y*x), (y*x)\cdot y)
=\big(y,x\t (y\cdot y)\big)
\]
and so
\[
\big(y,x\big)=
\sigma\Big(
y*\big(x\t^{-1}(y\cdot y)\big)
,
 \big(y*(x\t^{-1}(y\cdot y)\big)\cdot y
 \Big)
\]
\[
=
\sigma\Big(
(y*x)\t^{-1} y
,
 \big((y*x)\t^{-1} y\big)\cdot y
 \Big)
\]
From the above equality one can deduce -and easily check- 
the formula for $\sigma ^{-1}$:

\[
\sigma^{-1}(x,y)=
\Big(
(x*y)\t^{-1} x
,
 \big((x*y)\t^{-1} x\big)\cdot x
 \Big)
\]

\end{proof}

\subsection{About right non-degeneracy and the square map}

For {\em involutive solutions} (where the associated rack is trivial) and right-non degeneracy, Rump proves the following:

\begin{prop}\cite[Proposition 2]{R}
Let $\sigma$ be an involutive and 
 left-non degenerate solution of the Braid equation, then 
 it is right non-degenerate if and only if the map 
 $x\mapsto x\cdot x$ is bijective.
\end{prop}

In the general (i.e. non neccessarily involutive case)
we have the following

\begin{teo}\label{x2}
Let $\sigma$ be a left non-degenerate solution
given by
\[
\sigma(x,y)=\Big(
x*y,(x*y)\cdot \big(x \t (x*y)\big)
\Big)
\]
where
$\cdot$ and $\t$  are
 two operations satisfying conditions 1,2,3 of Theorem \ref{xoverrack}.
Assume in addition that $X$ is a finite set. Then
$\sigma$ is  right non-degenerate if and only if the map 
\[
x\mapsto x\cdot x
\]
 is bijective.
\end{teo}

This theorem is a consequence of the following two results, 
most of the parts work for non necessarily finite $X$:

\begin{prop}Let $(X,\cdot,\t)$ be a set together with two operations satisfying conditions 1,2,3 of Theorem \ref{xoverrack}.
If the $(X,\t)$ is a rack (that is,  $(-)\t y$ is bijective for all $y\in X$) 
and the map 
\[
X\longrightarrow X\hskip 1cm \]
\[
 x\mapsto x\cdot x=:x^2
 \]
 is bijective, then, for every $y$
 the map
\[
X\longrightarrow X\hskip 1cm 
 \]
 \[
x\mapsto (x*y)\cdot (x\t (x*y))\]
is bijective. That is, the corresponding solution of the Braid equation is also right-non degenerate.
\end{prop}
\begin{proof}
Using the identity
\[
(a\cdot b)\cdot(a\cdot c)=(b\cdot (a\t b))\cdot (b\cdot c)
\]
with  $c=a\t b$
we get
\[
(b\cdot (a\t b))\cdot (b\cdot (a\t b))
=(a\cdot b)\cdot (a\cdot (a\t b))
\]
If  $b=x*y$ and  $a=x$ then
\[
\Big((x*y)\cdot \big(x\t (x*y)\big)\Big)\cdot 
\Big((x*y)\cdot (x\t (x*y))\Big)
=\]
\[
=\big(x\cdot (x*y)\big)\cdot \Big(x\cdot \big(x\t (x*y)\big)\Big)
=y\cdot (x^2\t y)
\]
Notice
 $z\mapsto y\cdot (z\t y)$ is bijective
 (with inverse $w\mapsto (y*w)\t^{-1}y$). Since we assume
  $x\mapsto x^2$ bijective, in order to see that
 \[
x\mapsto  (x*y)\cdot (x\t (x*y))
\] 
is bijective, it is enough to see
that the following map is bijective
 \[
x\mapsto  \Big((x*y)\cdot (x\t (x*y))\Big)^2
\] 
But the previous computation gives
\[
x\mapsto 
 \Big((x*y)\cdot (x\t (x*y))\Big)^2
=y\cdot (x^2\t y)
\]
we conclude that it is a composition of bijective maps, hence, bijective.
\end{proof}

\begin{prop}Let  $\sigma:X\times X\to X\times X$ be a
 bijective and left-non degenerate solution of the Braid equation. Assume  it is also right non-degenerate, that is
for all fixed $y$,
the map
\[
X\longrightarrow X\hskip 1cm 
 \]
\[
x\mapsto (x*y)\cdot (x\t (x*y))
\]
is bijective,
and call $y\circ z$ the inverse operation. That is
\[
y\circ z=x \iff z =(x*y)\cdot (x\t (x*y))
\]
Then
\begin{equation}
 x^2\circ \vphi(x^2)=x
 \end{equation}
 where $\vphi(a)=a\t a$.
 In particular the map
 $x\mapsto x^2$ is injective. If $X$ is finite then it is also bijective.
 \end{prop}

\begin{proof}
 From the operational point of view
we can use the definition
\[
y\circ z=x \iff z =(x*y)\cdot (x\t (x*y))
\]
and compute, for $y=x ^2$ and $z=\vphi(x ^2)$,
giving
\[ x^2\circ \vphi(x^2)=x\iff
\vphi(x ^2)=
(x*x ^2)\cdot (x\t(x*x ^2))
\]
Clearly $x*x^2=x*(x\cdot x)=x$, so
\[ x^2\circ \vphi(x^2)=x\iff
\vphi(x ^2)=
x\cdot (x\t x)
\]
but using $\vphi(a)=a\t a$ and $a\cdot\vphi(x)=\vphi(a\cdot x)$
for all $a$ and $x$, we easily get
\[
x\cdot (x\t x)=x\cdot \vphi(x)=\vphi(x\cdot x)=\vphi(x^2)
\]
as desired. Also, from the diagramatic point of view,
\[
\xymatrix{
x\ar[rd]|\hole&y\ar[ld]\\
x*y&z=(x*y)\cdot (y\t (x*y))
}
\iff
\xymatrix{
y\circ z\ar[rd]|\hole&y\ar[ld]\\
(y\circ z)*y&z
}
\]
or equivalently
\[
\xymatrix{
x\ar[rd]|\hole&x\cdot y\ar[ld]\\
y&z=y\cdot (x\t y)
}
\iff
\xymatrix{
(x\cdot y)\circ z\ar[rd]|\hole&x\cdot y\ar[ld]\\
((x\cdot y)\circ z)*(x\cdot y)&z=y\cdot (x\t y)
}
\]
Taking $x=y$ gives
\[
\xymatrix{
x\ar[rd]|\hole&x ^2\ar[ld]\\
x&z=x\cdot (x\t x)
}
\iff
\xymatrix{
(x\cdot x)\circ z\ar[rd]|\hole&x\cdot x\ar[ld]\\
((x\cdot x)\circ z)*x ^2&z=x\cdot (x\t x)
}
\]
That is,
\[
\xymatrix{
x\ar[rd]|\hole&x ^2\ar[ld]\\
x&z=\vphi(x ^2)
}
\iff
\xymatrix{
x ^2\circ z\ar[rd]|\hole&x ^2\ar[ld]\\
(x ^2\circ z)*x ^2&z=\vphi(x ^2)
}
\]
so, $x ^2\circ z=x ^2\circ\vphi(x ^2)=x$.

In any case, we conclude that the map $x\mapsto x^2$ is injective.
\end{proof}

\begin{rem}
If $X$ is a finite set, then the bijectivity of $x\mapsto x^2$
 follows from injectivity.
\end{rem}

\section{Twisting by automorphisms and
linear automorphisms}

We call a bijective map $k:X\to X$ an automorphism
of $\sigma$ id $(k\times k) \sigma=
\sigma(k\times k)$.
We leave to the reader the proof of the following:

\begin{lem}Let $(X,\sigma)$ be a birack determined by the operations
$\cdot$ and $\t$. A bijective map $k:X\to X$ is an automorphism for 
$\sigma$ if and only if 
\[
k(x\cdot y)=k(x)\cdot k(y),\ \ \ 
k(x\t y)=k(x)\t k(y),\ \ \forall x,y\in X\] 
\end{lem}

\subsection{First kind of twist}

Let  $\sigma:X\times X\to X\times X$ be a birack and
 $k:X\to X$
an automorphism.

\begin{defi}
 We define a new bijective map $X\times X\to X\times X$
  by
\[
\sigma_k:=
 (k^{-1}\times \id)\sigma(k\times \id)
=
(\id \times k)\sigma(\id\times k^{-1})
\]
\end{defi}

\begin{rem}
$\sigma_k$ is also a solution of the braid equation.
\end{rem}

On elements, the new map $\sigma_k$ is given by
\[
\sigma_k(x,y)
=(k ^{-1}\times \id)\sigma(k(x),y)
\]
\[
=(k ^{-1}\times \id)\Big(k(x)*y,
\big( k(x)*y\big)\cdot \big( k(x)\t (k(x)*y) \big)
\Big)
\]
\[
=\Big(x*k^{-1}(y),
\big( k(x)*y\big)\cdot \big( k(x)\t (k(x)*y) \big)
\Big)
\]
It is then natural to consider the new operation
$\*k$ dekined by
\def\*k{*_k}
\def\ck{\cdot_k}
\[
x*_k y:=x*k^{-1}(y),
\ \hbox{ with inverse } \
x\cdot_k y=
k(x\cdot y)\]
We clearly see that the first coordinate of $\sigma_k(x,y)$
is $x\*k y$, but if we compute
$(x\*ky)\ck (x\t(x\*ky))$ we get
\[
(x\*k y)\ck \big(x\t(x\*ky)\big)=
k\Big(\big(x*k ^{-1}(y)\big)\cdot  \big(x\t(x*k^{-1}(y))\big)
\Big)\]
\[
=
(k(x)*y)\cdot \Big(k (x)\t \big(k(x)*y)\big)\Big)
\]
that is, precisely  the second coordinate. 
We have proven the following:

\begin{lem}If $\sigma=\sigma_{(\cdot,\t)}$ and $k:X\to X$
is an automorphism of $\sigma$, then
\[
(\sigma_{(\cdot,\t)})_k
=
(k^{-1}\times k)\circ \sigma_{(\cdot,\t)}\circ(k\times \id)
=\sigma_{(\cdot_k,\t)}
\]
That is, the twisting procedure changes the dot operation but
preserves the rack one.
 \end{lem}

 \begin{rem}\label{autotwist}
 If $k:X\to X$ is an automorphism of $\sigma$, then it is also an automorphism of $\sigma_k=
 (k\times k^{-1})\circ\sigma
$, because $\sigma_k$ clearly commutes with $(k\times k)$ if
$\sigma$ does.
 \end{rem}
 
  \subsection{Skew-racks as twisting of racks}
  
Let $(X,\t)$ be a rack and consider the corresponding
solution $\sigma(x,y)=(y,x\t y)$. That is, we consider the trivial $*$ operation
\[
x*y=y,\ \forall x,y\in X\]
If $k:X\to X$, then trivially $k$ satiskies
\[
k(x\cdot y)=k(y)=x\cdot k(y)=k(x)\cdot k(y)
\]
So, the condition of being an automorphism of $\sigma$
 is equivalent to be an automorphism of the rack $(X,\t)$.  Then, the formula
 \[
 \sigma_{k^{-1}}(x,y)=
 (k\times \id)\sigma(k^{-1}\times \id)(x,y)
=\big(k(y), k^{-1}(x)\t y\big)
 \]
 is a solution of the Braid equation. Notice that the operation
 \[
 x\bt y:=k^{-1}(x)\t y
 \]
 is not necesarily a rack operation, because
 \[
 (x\bt y)\bt z:
 =k^{-1}\big(k^{-1}(x)\t y\big)\t z
 =\big(k^{-2}(x) \t k^{-1}(y)\big)\t z
 \]
 \[
 =\big(k^{-2}(x)\t z) \t (k^{-1}(y)\t z)
 \]
 \[
 =\big(k^{-1}(x)\bt z) \t (y\bt z)
 \]
 \[
 =\big(x\bt k(z)\big) \bt (y\bt z)
 \]
 However, $k$ is also an automorphism for the operation $\bt$ (see Remark \ref{autotwist}).
 Recall
   the notion of skew-racks, used in 
 \cite{N}, (we will use a black triangle for skew racks
 and reserve usual triangles for usual racks).
 
\begin{defi}  A triple $(X,\bt,k)$, where $X$ is a set, 
  $\bt$ is a binary operation on $X$, and
  $k:X\to X$ is a bijective map, verifying
 \begin{itemize}
 \item $(-)\bt  x:X\to X$
  is bijective for all $X$ and
\item $ k(x\bt y)=k(x)\bt k(y)$
\item $ (x\bt y)\bt z
 =\big(k(x)\bt z) \t (y\bt z)
 $
 \end{itemize}
 is called a {\bf skew-rack}.
 \end{defi}
 Given a skew rack $(X,\bt,k)$,
 one can easily check that
 \[
 \sigma(x,y):=(k(y),x\bt y)
 \]
 is a solution of the Braid equation.
   We saw that if $(X,\t)$
  is a rack and $k$ is a rack automorphism of $X$,
  then the operation $x\bt y:=k ^{-1}(x)\t y$, together with $k$
  is a skew rack structure on $X$. But conversely,
  if $(X,k,\bt)$ is a skew rack, then its associated rack is
  \[
  x\t y := k(x)\bt y
  \]
We conclude that if we twist a skew rack by $k$
then we obtain a rack solution, and $k$ is an automorphism of the rack (hence, an automorphism of the solution) and
the original skew-rack comes from twisting (by $k^{-1}$)
a usual rack solution.

\subsection{Second kind of twist: linear automorphisms}

We begin with a characterization of a special
family of automorphisms:

\begin{lem}
Let $\phi:X\to X$ be a bijection.
\begin{enumerate}
\item $\sigma(\id\times \phi)=(\phi\times\id) \sigma$
if and only if 
\[
\left\{
\begin{array}{rcl}
\phi(x\cdot y)&=&x\cdot \phi(y)
\\
y\cdot (x\t y)
&=&\phi(y)\cdot \big(x\t \phi(y))
\end{array}\right.
 \]
\item $\sigma(\phi\times \id)=(\id\times \phi) \sigma$
if and only if
\[
\left\{
\begin{array}{rcl}
x\cdot y&=&\phi(x)\cdot y\\
y\cdot \big(\phi(x)\t y  )
&=&\phi\big(y\cdot(x\t y)\big)
 \end{array}\right.
 \]

\item $\phi$ satisfies 
$\sigma(\id\times \phi)=(\phi\times\id) \sigma$
and $\sigma(\phi\times \id)=(\id\times \phi) \sigma$
if and only if

\[
\left\{
\begin{array}{rclrcl}
\phi(x\cdot y)&=&x\cdot \phi(y),&
\phi(x)\cdot y&=&x\cdot y
\\
\\
\phi(x\t y)  &=& \phi(x)\t y,&
x\t \phi(y)&=&x\t y
 \end{array}\right.
\] 
 \end{enumerate}
\end{lem}

\begin{proof}In order to prove 1 
we check the equality
$\sigma(\id\times \phi)=(\phi\times\id) \sigma$
on pairs of the form $(x,x\cdot y)$:
\[
(\phi\times\id) \sigma(x,x\cdot y)=
(\phi\times\id) (y,y\cdot (x\t y))=
(\phi(y),y\cdot (x\t y))
\]
On the other hand,
\[
\sigma(\id\times \phi)(x,x\cdot y)=
\sigma(x,\phi(x\cdot y))
\]
\[=
\Big(x*\phi(x\cdot y),
\big(x*\phi(x\cdot y)\big)\cdot \big(x\t (x*\phi(x\cdot y))
\big)
\Big)=
\]
so the equality holds if and only if both coordinates are equal.
From the first coordinate we get
$x*\phi(x\cdot y)=\phi(y)$, or equivalently
$\phi(x\cdot y)=x\cdot \phi(y)$. For the second coordinate
\[
y\cdot (x\t y)
=\big(x*\phi(x\cdot y)\big)\cdot \big(x\t (x*,\phi(x\cdot y))
\big)
\]
using $x*\phi(x\cdot y)=\phi(y)$, we get
that the equality holds if and only if
\[
y\cdot (x\t y)
=\phi(y)\cdot \big(x\t \phi(y))
\]

To prove 2 we proceed similarly, 
we check on elements of the form $(x,x\cdot y)$:

\[
\sigma(\phi\times \id)(x,x\cdot y)=\sigma(\phi(x),x\cdot y)
\]
\[
=\Big(
\phi(x)*(x\cdot y),
\big(\phi(x)*(x\cdot y)\big)\cdot \big(\phi(x)\t ( \phi(x)*(x\cdot y)  )\big)
\Big)
\]
while
\[
(\id\times \phi) \sigma(x,x\cdot y)=
(\id\times \phi) (y,y\cdot(x\t y))=
 \Big(y,\phi\big(y\cdot(x\t y)\big)\Big)
 \]
 So, the equality
 $\sigma(\phi\times \id)=(\id\times \phi) \sigma$ is equivalent to

\[
\left\{
\begin{array}{rcl}
\phi(x)*(x\cdot y)&=&y\\
\big(\phi(x)*(x\cdot y)\big)\cdot \big(\phi(x)\t ( \phi(x)*(x\cdot y)  )\big)
&=&\phi\big(y\cdot(x\t y)\big)
 \end{array}\right.
 \]
 which is clearly equivalent to
\[
\left\{
\begin{array}{rcl}
x\cdot y&=&\phi(x)\cdot y\\
y\cdot \big(\phi(x)\t y  )
&=&\phi\big(y\cdot(x\t y)\big)
 \end{array}\right.
 \]
 Part 3 is a consequence of 1 and 2 and some elementary manipulation. for example, using the second equality of 2
 and the first equality of 1
 
\[
 y\cdot \big(\phi(x)\t y  )
=\phi\big(y\cdot(x\t y)\big)
=y\cdot \phi(x\t y)
\] and since $y\cdot (-)$ is bijective 
we get $\phi(x)\t y = \phi(x\t y)$. The other equalities are similar.

\end{proof}

\begin{defi}
We say that a bijection $\phi:X\to X$ is a {\bf linear automorphism}
if and only if it satisfies condition 3, namely
\[
\left\{
\begin{array}{rclrcl}
\phi(x\cdot y)&=&x\cdot \phi(y),&
\phi(x)\cdot y&=&x\cdot y
\\
\\
\phi(x\t y)  &=& \phi(x)\t y,&
x\t \phi(y)&=&x\t y
 \end{array}\right.
\] 
\end{defi}

In terms of diagrams, $\phi$ is a linear automorpfism if
the following equalities hold:
\[
\xymatrix@-2ex{
\ar[rrdd]|\hole&&\ar[lldd]|(0.3){\phi}&&\ar[rrdd]|\hole&&\ar[lldd]|(0.7){\phi}
\\
&&&=&&&&\hbox{,}\\
&&&&&&
}\hskip 1cm 
\xymatrix@-2ex{
\ar[rrdd]|(0.25){\phi}|(0.51){\hole}&&\ar[lldd]&&
\ar[rrdd]|(0.48){\hole}|(0.71){\phi}&&\ar[lldd]
\\
&&&=&&&\\
&&&&&&
}\]

\begin{rem}
A linear automorphism is, in particular, an automorphism because
\[
\sigma (\phi\times \phi)
=
\sigma (\phi\times \id) (\id \times \phi)
\]
\[=
(\id \times \phi)
 \sigma (\id \times \phi)
=
(\id \times \phi) (\phi\times \id)
 \sigma=(\phi\times \phi) \sigma\]
 Or, in diagrams, it is clear that the equality
 \[
\xymatrix@-2ex{
\ar[rrdd]|(0.25){\phi}|(0.51){\hole}&&\ar[lldd]|(0.3){\phi}&&\ar[rrdd]|(0.48){\hole}|(0.71){\phi}&&\ar[lldd]|(0.7){\phi}
\\
&&&=&&&&\\
&&&&&&
}\]
is a consequence of the set of two diagramas below.
\end{rem}

\begin{defi}
With the same notations as before, denote
\[
\sigma^\phi:=\sigma(\id\times \phi)
\]
\[
{}^\phi\sigma:=\sigma(\phi\times\id)
\]
\end{defi}

\begin{rem}
Both $\sigma^\phi$ and ${}^\phi\sigma$ are new solutions of the
braid equation. For example
for $\sigma^\phi$ we look at the diagram
corresponding to
$
(\sigma ^\phi\times \id)(\id\times \sigma ^\phi)(\sigma^\phi\times \id)$
and get
\[
\xymatrix{
\ar[rd]|\hole&\ar[ld]|(0.4){\phi}&\ar@{=}[d]\\
\ar[d]& \ar[rd]|\hole&\ar[ld]|(0.4){\phi}&=&\\
\ar[rd]|\hole& \ar[ld]|(0.1){\phi}& \ar@{=}[d]\\
&&
}
\xymatrix{
\ar[rd]|\hole&\ar[ld]|(0.2){\phi}&\ar[d]|(0.1){\phi ^2}\\
\ar[d]& \ar[rd]|\hole&\ar[ld]\\
\ar[rd]|\hole& \ar[ld]& \ar@{=}[d]\\
&&
}
\]
and one can use the Braid equation for $\sigma$ and compare with
$
(\id\times \sigma ^\phi)
(\sigma ^\phi\times \id)(\id\times \sigma ^\phi)$.
The argument for ${} ^\phi\sigma$ is the same.
\end{rem}

\begin{rem}\label{linear}
Since $\phi$ clearly commutes with $\phi$ and $\phi^{-1}$,
it is clear that $\phi$ (and $\phi^{-1}$) is again a linear
 automorphism for both ${} ^{\phi}\sigma$ and
$\sigma ^\phi$.
\end{rem}

\subsection{Twisting by linear automorphism and operations}

In terms of the operations $\cdot$ and $\t$ we have,
for ${} ^\phi\sigma$:
\[
{}^\phi\sigma(x,y)=\sigma\big(\phi(x),y\big)
=\Big(
\phi(x)*y
,
(\phi(x)*y)\cdot \big(\phi(x)\t(\phi(x)*y)\big)
\Big)
\]
\[
=\Big(
x*y
,
(x*y)\cdot \big(\phi(x)\t(x*y)\big)
 \Big)
\]
and so, the operations $*$ and $\cdot$ do not change, but
the rack operation $\t$ changes into
$\t_\phi$ where
\[
x\t_\phi y =\phi(x)\t y
\]

Recall the general notation for twisting operations
under an automorphism, and in particular under a linear
automorphism $\phi$:
\[
 x\cdot_\phi y=x\cdot \phi(y)= \phi(x\cdot y),
 \ \ \ 
x*_\phi y=x*\phi^{-1}(y),
\ \ \  
x\t_\phi y =\phi(x)\t y\]

We have proven  the first part
of the following Proposition:

\begin{prop}
Let $\sigma=\sigma_{(\cdot ,\t)}$ denote the solution 
associated to 
a pair of operations $\cdot$ and $\t$
and let $\phi:X\to X$ be a linear automorphism. 
Then, the twisted solutions can be described by
\[
{}^\phi\Big(\sigma_{(\cdot,\t)}\Big)=
\sigma_{(\cdot,\t)}(\phi\times \id)
=\sigma_{(\cdot,\t_\phi)}
\]
\[
\Big(\sigma_{(\cdot,\t)}\Big)^{\phi^{-1}}=
\sigma_{(\cdot,\t)}(\id\times\phi)=
\sigma_{(\cdot_\phi,\t_{\phi^{-1}})}
\]

\end{prop}
\begin{proof}
Let us compute  $\sigma ^\phi$ on elements:

\[
\Big(\sigma_{(\cdot,\t)}\Big)^{\phi^{-1}}(x,y)
=\sigma_{(\cdot,\t)}(x,\phi^{-1}(y))
\]
\[
=\Big(
x*\phi^{-1}(y),
(x*\phi^{-1}(y))\cdot\big(
x\t(x*\phi^{-1}(y))
\big)
\Big)
\]
\[
=\Big(
x*_{\phi} y,
(x*_{\phi} y)\cdot\big(
x\t(x*_{\phi} y)
\big)
\Big)
\]
\[
=\Big(
x*_{\phi} y,
(x*_{\phi} y)\cdot_ \phi \phi ^{-1}\big(
x\t(x*_{\phi} y)
\big)
\Big)
\]
\[
=\Big(
x*_{\phi} y,
(x*_{\phi} y)\cdot_ \phi 
\big(\phi ^{-1}(x) \t(x*_{\phi} y) \big)
\Big)
\]
\[
=\Big(
x*_{\phi} y,
(x*_{\phi} y)\cdot_ \phi 
\big(x \t_ {\phi^{-1}}(x*_{\phi} y) \big)
\Big)
=\sigma_{(\cdot_\phi,\t_{\phi^{-1}})}(x,y)\]
\end{proof}

\subsection{The intrinsic  maps $\vphi$ and $i$}

If $(X,\t)$ is a rack, we denote
 $\vphi(x):=x\t x$. It is well-known that
$\vphi$ is a bijection, with inverse $i(x)=x\t^{-1}x$, and that $\vphi$ (and $i$) satisfy
\[
\vphi(x\t y)  =\vphi(x)\t y,
\hskip 1cm
x\t \vphi(y)=x\t y
\]
The map $i:X\to X$ is also denoted $\Tw$.

\begin{prop}
Assume $(X,\cdot,\t)$ is a finite birack, in particular
$(X,\t)$ is a rack. Define $\vphi$ 
(and $i$) as before.
Then $\vphi$ (and its inverse $i$) is
a linear automorphism of the birack.
\end{prop}

\begin{proof}
We need to prove the equalities
\[
\vphi(x\cdot y)\overset{?}{=}x\cdot \vphi(y)
,\ \ \ 
\vphi(x)\cdot y\overset{?}{=}x\cdot y
\]
The first one is easy: because 
$x\cdot (-)$ is a rack morphism, then
\[
x\cdot \vphi(y)=x\cdot (y\t y)=(x\cdot y)\t (x\cdot y)
=\vphi(x\cdot y)
\]
For the second equality, we use the identity
\[
(a\cdot b)\cdot(a\cdot c)=(b\cdot (a\t b))\cdot (b\cdot c)
\]
for 
$a=b$  and get
\[
(a\cdot a)\cdot(a\cdot c)=(a\cdot (a\t a))\cdot (a\cdot c)
\]
That is
\[
a ^2\cdot(a\cdot c)=(a ^2\t a^2)\cdot (a\cdot c)
\]
or
\[
a ^2\cdot(a\cdot c)=\vphi(a ^2)\cdot (a\cdot c)
\]
If $c=a* y$ we get
\[
a ^2\cdot y=\vphi(a ^2)\cdot y
\]
for all $y$ and $a$. But because $a\mapsto a ^2$ is bijective,
we get $x\cdot y=\vphi(x)\cdot y$ for all $x,y$.
\end{proof}

\begin{rem}
The hypotesis of $X$ being finite is only used to be sure that $x\mapsto x^2$ is bijective.
\end{rem}

\begin{coro}
Every finite birack is a twist (of the second kind) of a biquandle.
\end{coro}

\begin{proof}
If $(X,\t)$ is a rack, then
$(X,\t_i)$ is a quandle, because
\[
x\t_i x=i(x)\t x=(x\t ^{-1}x)\t x=x\]
and so
\[
{}^i\sigma=\sigma_{(\cdot,\t_i)}
\]
is a biquandle, because $(X,\t_i)$ is a quandle. But clearly
(see Remark \ref{linear})
\[
\sigma={} ^\vphi({}^i\sigma)
\]
So, the original $\sigma$ is a twist of a biquandle.
\end{proof}

 \section{Two applications of the bijectivity of $x\mapsto x\cdot x=:x^2$}
 \subsection{A remark on the envelopping group}
 
 Recall that given $\sigma:X\times X\to X\times X$ there is a group
 attached to it, called envelopping group and usually 
 denoted $G_X$, defined by
 \[
 G_X:=\frac{Free(X)}{\langle xy=zt : x,y\in X,
 \sigma(x,y)=(z,t)
 \rangle}
\]
For example if $\sigma(x,y)=(y,x\t y)$ then we have
\[
xy =y(x\t y) \ \in G_X
\]
and so 
\[
x\t y=y^{-1}xy \in G_X.
\]
For a general solutions, the relations should be written in terms of the choosen notation for $\sigma(x,y)$. Using
our proposal, we have
 we have the following equality
in $G_X$:
\[
xy =\Big(
x*y
\Big)
\Big(
\big(x*y\big)
\big(x\t (x*y)\big)
\Big)
\]
Under the usual bijection $(x,y)\leftrightarrow(x,x\cdot y)$
one can translate the above relations into
 \[
 x(x\cdot y)=y (y\cdot (x\t y))\in G_X, \ \ \forall x,y\in X
 \]
This expression is explicit and clear enough in order
to prove  the following property of $G_X$:

\begin{prop}
Let $\sigma:X\times X\to X\times X$ be a finite birack,
$\t$ the derived rack operation, and $\sim$ the equivalence
relation on $X$
generated by 
\[
x\sim x\t x=\vphi (x)
\]
Then $G_X=G_{X/\sim}$.
\end{prop}

\begin{proof}
From the general equality in $G_X$ 
 \[
 x(x\cdot y)=y (y\cdot (x\t y))\in G_X, \ \ \forall x,y\in X
 \]
take in particular $x=y$, and get
 \[
 x(x \cdot x)=x(x\cdot \vphi(x))=x\vphi(x\cdot x)
 \]
 hence
 \[
 x \cdot x=\vphi  (x\cdot x)\in G_X.
 \]
And because  $x\mapsto x\cdot x$ is bijective, we conclude
 \[
 x=\vphi(x) \in G_X
 \]
\end{proof}

 \subsection{\label{sectionnelson}
 Using biracks for painting oriented
framed links and Nelson's condition} 
 
 In a similar way that one may use quandles for coloring
 oriented knots or links and racks for {\em framed}
  knots and links,
 one can use  biquandles to color knots and links and biracks for framed knots or links. In the case of framed knots/links one has the second and third Reidemeister move, but the first Reidemeister moves should be replaced by the modified ones.
 Here is a picture of one of them, the other is its mirror image:
 
 \includegraphics[width=8cm]{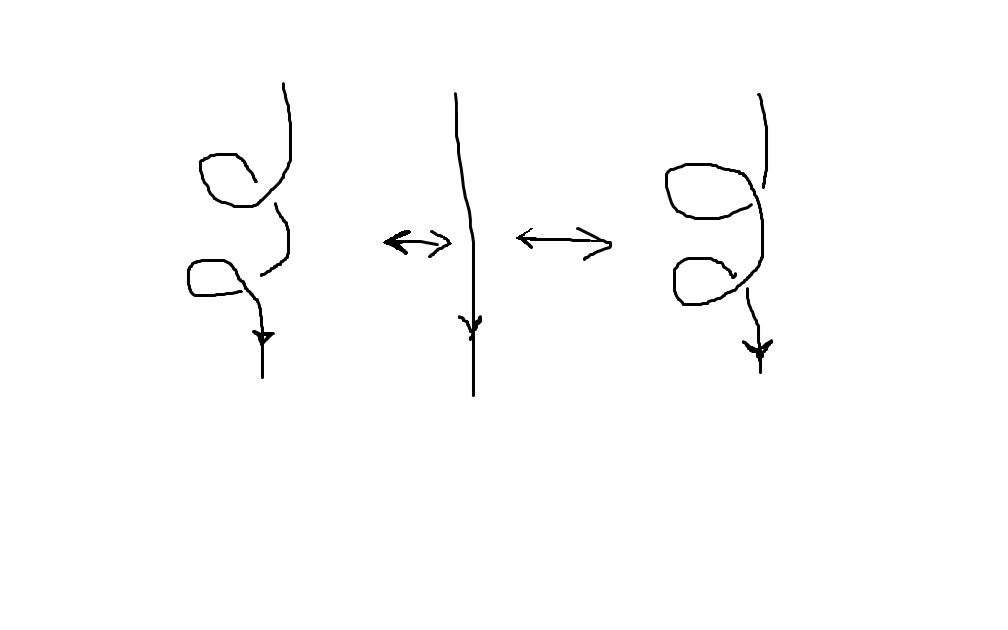}
 
 In case or racks, there is only one choice for coloring such a
 diagram begining with color $x$:
 
  \includegraphics[width=3.4cm]{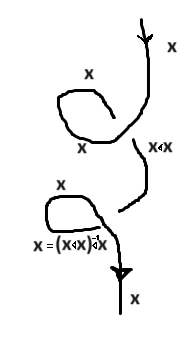}
 
For biracks, the situation doesn't look so easy, one first have 
 to solve the problem of coloring the top part of the diagram:

 \includegraphics[width=4cm]{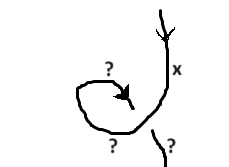}

 Recall the coloring rule now is
 \[
 \xymatrix{
a\ar[rd]|\hole&c\ar[ld]\\
a*c&(a*c)\cdot (a\t(a* c)) 
 }\]
 or equivalently
 \[
 \xymatrix{
a\ar[rd]|\hole&a\cdot b\ar[ld]\\
b&b\cdot (a\t b) 
 }\]
 So, if one wants $a=b$, then we need  $a$ such that $a\cdot a=x$:
 \[
 \xymatrix{
\ar@{=}[d]a\ar[rd]|\hole&a\cdot a = x\ar[ld]\\
a&a\cdot (a\t a)\hskip -1cm&\hskip -1cm =
\hskip -1cm&\hskip -0.5cm( a\cdot a)\t (a\cdot a)=\phi(x) 
 }\]
The question is, given $x$, can we found (a unique?) $a$
such that
$a\cdot a=x$?
 In \cite{Nelson}, the author propose a definition of birack as a 
 solution of the Braid equation that is, in principle, stronger than non-degenerate. Recall the definiton of {\em strongly invertible} with notation as in \cite{Nelson}:
 
\begin{defi}\cite[Definition 1]{Nelson}\label{defnelson}
 We will say a map $B : X \times X \to X \times X$ is strongly invertible provided $B$ satisfies the following three conditions:
\begin{itemize}
    \item $B$ is invertible, i.e., there exists a map $B^{-1} : X \times X \to X \times X$ satisfying $B \circ B^{-1} = \mathrm{Id}_{X \times X} = B^{-1} \circ B$,
    \item $B$ is sideways invertible, i.e., there exists a unique invertible map $S : X \times X \to X \times X$ satisfying
    \[
    S(B_1(x, y), x) = (B_2(x, y), y),
    \]
    for all $x, y \in X$, and
    \item the sideways maps $S$ and $S^{-1}$ are diagonally bijective, i.e., the compositions $S^{\pm 1}_1 \circ \Delta$, $S^{\pm 1}_2 \circ \Delta$ of the components of $S$ and $S^{-1}$ with the map $\Delta : X \to X \times X$ defined by $\Delta(x) = (x, x)$ are bijections.
\end{itemize}
\end{defi}
With our notation $B=\sigma$, $B_1(x,y)=x*y$, $B_2(x,y)=
(x*y)\cdot (x\t (x*y))$. Notice that  the second item of the definition means, looking at the diagram
\[
 \xymatrix{
x\ar[rd]|\hole&y\ar[ld]\\
B_1(x,y)&B_2(x,y), 
 }\]
that the left side of the diagram determines (using $S$) 
the right side, and viceversa. While the ``diagonally bijective condition''
means that the right hand side of this diagram 
\[
 \xymatrix{
a\ar[rd]|\hole&?\ar[ld]\\
a&? 
 }\]
 is uniquely determined by $a$, and viceversa.

It is not clear in \cite{Nelson} that the diagonally bijective condition on Definition \ref{defnelson}
is or is 
not a consequence of the left and right non-degeneracy, 
and this diagonally bijection is required in  \cite{Nelson}
as part of the
 definition
of  a birack. But we know that
\[
\sigma(x,x\cdot x)=(x, x\cdot(x\t x))=(x,\vphi(x\cdot x))
\]
or in diagram:

\[
 \xymatrix{
x\ar[rd]|\hole&x ^2\ar[ld]\\
x&\vphi(x ^2) 
 }\]

So, the bijectivity of Nelson's diagonal map
is equivalent to the bijectivity
of the map $x\mapsto x^2$, that we now know is a consequence
of right non-degeneracy (Theorem \ref{x2}).
Denote $\sqrt{\ }$ the inverse of $x\mapsto x\cdot x$.
Now it is clear that

 \[
 \xymatrix{
\sqrt{x}\ar[rd]|\hole&\sqrt{x}\cdot \sqrt{x}=x\ar[ld]\\
\ar@/^5ex/[u]\sqrt{x}&x\t x=\vphi(x)
 }\]
gives the unique coloring of that diagram with color $x$ 
in the top (right), and it is easy to check that using the inverse of $\vphi$
one can also color the botton part of the modified Reidemeister move, and coloring with biracks gives a rule that is compatible with modified first Reidemesiter move.
 We conclude that (finite)
 biracks (in the sense of bijective, and left and right non-degenerate solutions) are good objects to color
framed knots/links because they are compatible with modified
 Reidemeister moves, without any additional assumption
of ``diagonally bijective'' property.

As a side remark, if the ``Twisting'' map $X\to X$
corresponds
to a coloring rule using the twisting picture, then the twisting
map associated to a birack {\em is the same} as the twisting map
given by the associated (or derived) rack. In 
\cite{Nelson}, the twisting map is called ``kink map''.

 \[
 \xymatrix{
\ar@/ ^1ex/[rd]|(0.56){\hole}&x\ar@/^0.5ex/[ld]\\
 \ar@/^5ex/[u]&\Tw(x),&\Tw(x)=x\t x
 }\]

\section{\label{braces}
Comparison with Guarnieri-Vendramin's  skew-braces}

We end with a comparaison of the general situation and
the solution given by skew-braces.
Recall from \cite{GV} the definition of skew-brace.
We warn the reader that in this section we are going to use 
{\em additive} notation  
for {\em non necessarily commutative} groups.

\begin{defi}\cite[Definition 1.1]{GV}
A skew-brace structure on a set $A$ is the data of
two simultaneous group structures $(A,+)$ and $(A,\circ)$
satisfying
\[
a\circ(b+c)=a\circ b-a+a\circ b
\]
where we adopt the convention that the $\circ$ operation is
performed first, that is
\[
a\circ b-a+a\circ b:=(a\circ b)-a+(a\circ b)
\]
\end{defi}
A skew-brace with $(A,+)$
a {\em commutative} group is a (usual) brace.
 In 
\cite{GV} it is showed that for a skew brace $(A,+,\circ)$,
the formula

\[
r(x,y)=\Big(-x+x\circ y,(-x+x\circ y)'\circ x\circ y\Big)
\]
gives a solution of the Braid equation, where $a'$ is the inverse with
respect the circle operation.
This generalizes the situation of usual braces
considered by Rump:
 when
 $(A,+)$ is 
{\em commutative} one gets an {\em involutive} solution.

In our notation, $a*c$ corresponds to the first coordinate
of the solution, that is
 \[
 a*c:=-a+a\circ c\]
and so the inverse (dot) operation is computed as
 \[
a\cdot b=c\iff  a*c=b
\iff -a+a\circ c=b
\]
 \[
\iff c=a'\circ(a+b)\]
So, we get $a\cdot b=
a'\circ(a+b)$, hence
\[
r(x,x\cdot y)=r\big(x,x'\circ(x+y)\big)=
\]
\[
=\Big(-x+x\circ \big(x'\circ(x+y)\big),
\big(-x+x\circ (x'\circ(x+y))\big)'\circ x\circ \big(x'\circ(x+y)\big)
\Big)
\]
Notice as expected
$
-x+x\circ (x'\circ(x+y))
=-x+(x+y)=y$, so
\[
r(x,x\cdot y)
=\Big(y,
y'\circ x\circ \big(x'\circ(x+y)\big)
\Big)
\]
But also $x$ and $x'$ cancels with the circle operation so
\[
r(x,x\cdot y)
=\big(y,
y'\circ (x+y)\big)
\]
This should  be equal to  $(y,y\cdot(x\t y))$ 
where $\t$ is the derived rack structure. We will compute 
$\t$ in terms of the braces.
Recall the dot operation is given by
\[
a\cdot b=a'\circ(a+b)
\]
So,
\[
y'\circ  (x+y) =y\cdot (x\t y) \iff
\]

\[
\iff
y'\circ  (x+y) =y'\circ\big(y +(x\t y) \big)
 \]
\[
\iff
x+y =y +(x\t y) 
 \]
\[
\iff
-y+x+y =x\t y
 \]
 That is, the solutions given by skew braces are the ones
 whose derived rack is a group with conjugation as rack
  operation:
 $\Conj(A,+)$.


\begin{thebibliography}{99}

\bibitem[D]{Anastasia}
A.  Doikou, {\em Parametric Set-theoretic Tang-Baxter equation:
 p-racks solutions  \& quantum algebras.}
 \url{https://arxiv.org/pdf/2405.04088}.


\bibitem[FG]{FG} M. Farinati, J. Garc\'ia Galofre, {\em Link 
and knot invariants from non-abelian Yang-Baxter 
2-cocycles},
Journal of Knot Theory and its Ramifications  Vol. 2,
 No. 13, (2016).

\bibitem[GV]{GV} L. Guarnieri, L. Verndramin, 
{\em 
Skew braces and the Yang-Baxter equation.}
Mathematics of Computation
Vol. 86, No. 307 (2017), pp. 2519-2534.


\bibitem [LYZ]{LYZ}
 J. Lu, M. Yan,  Y. Zhu, {\em On Set-theoretical Yang–Baxter equation}, 
Duke Math. J. 104 (2000), pp. 1-18.

\bibitem[Ne]{Nelson}S. Nelson, {\em  Link invariants from finite biracks}, 
Banach Center Publications (2014)
    Volume: 100, Issue: 1, page 197-212
    ISSN: 0137-6934


\bibitem[No]{N} T. Nosaka,
{\em 
Skew-rack cocycle invariants of closed 3-manifolds},\\
\url{https://arxiv.org/abs/2303.12995}

\bibitem[R]{R}W. Rump, {\em 
A decomposition theorem for square-free unitary
solutions of the quantum Yang-Baxter equation},
Adv. in Math. Vol 193, Issue 1, (2005), Pages 40-55. \\

\bibitem[SV]{AV} A. Smoktunowicz, L. Vendramin,
{\em On skew braces (with an appendix by N. Byott and L. Vendramin)},
J. of Comb. Algebra 2 (2018), no. 1, pp. 47-86.


\bibitem[S]{S} A. Soloviev, {\em 
Non-unitary set-theoretical solutions to the quantum Yang-Baxter equation}, Math. Res.
Lett. 7 (2000), no. 5-6, pp. 577-596


\end{thebibliography}
\end{document}